\pdfoutput=1
\RequirePackage[l2tabu, orthodox]{nag}

\documentclass[reqno,tbtags]{amsart}

\usepackage{fullpage} % FOR WIDE MARGINS

\usepackage{lmodern}
\usepackage[T1]{fontenc}
\usepackage[utf8]{inputenc}
\usepackage[english]{babel}
\usepackage{microtype} %Better font spacing

\usepackage{amsmath,amssymb,amsthm,mathrsfs,latexsym,mathtools,mathdots,
booktabs,array,enumerate,tikz,bm,comment,hyperref}

\usepackage[centertableaux]{ytableau}

\usetikzlibrary{cd} %For commutative diagrams

\usetikzlibrary{arrows,positioning}
\usepackage[capitalize,noabbrev]{cleveref}

\usepackage{todonotes}
\presetkeys%
    {todonotes}%
    {inline,backgroundcolor=yellow}{}

\newtheorem{theorem}{Theorem}

\newtheorem{lemma}[theorem]{Lemma}
\newtheorem{corollary}[theorem]{Corollary}

\theoremstyle{definition}

\newtheorem{example}[theorem]{Example}

\newcommand{\defin}[1]{\emph{#1}}

\newcommand{\setN}{\mathbb{N}}
\newcommand{\setZ}{\mathbb{Z}}

\newcommand{\xvec}{\mathbf{x}}

\newcommand{\thsup}{\textnormal{th}}

\newcommand{\symS}{S}
\newcommand{\SSYT}{\mathrm{SSYT}}

\newcommand{\Wcomp}{\mathrm{WCOMP}}

\newcommand{\wc}{\alpha} %Composition

\newcommand{\schurS}{\mathrm{s}}
\newcommand{\monom}{\mathrm{m}}

\DeclareMathOperator{\length}{\ell}

\DeclareMathOperator{\cyc}{cyc}

\DeclareMathOperator{\cryss}{{\tilde{s}}}

\title{Free action and cyclic sieving on skew semi-standard Young tableaux}

\author{Per Alexandersson}
\address{Dept. of Mathematics, Royal Institute of Technology, SE-100 44 Stockholm, Sweden}
\email{per.w.alexandersson@gmail.com}


\begin{document}

\begin{abstract}
In this note, we provide a short proof of Theorem 3.3 in the paper titled \emph{Crystals, semistandard tableaux and  cyclic sieving phenomenon}, by Y.-T.~Oh and E.~Park,
which concerns a cyclic sieving phenomenon on semi-standard Young tableaux.
We give a short proof of their result which extends to skew shapes. 
\end{abstract}

\maketitle

% \setcounter{tocdepth}{2}
% \tableofcontents

\section{Introduction}

The \defin{cyclic sieving phenomenon}, (CSP) was introduced 
by V.~Reiner, D.~Stanton and D.~White~\cite{ReinerStantonWhite2004}.
Given a set $X$, a cyclic group $C$ of order $n$ acting on $X$ and 
a polynomial $f(q)$, the triple $(X,C,f(q))$
is said to \defin{exhibit the cyclic sieving phenomenon} if 
for all $k \in \setZ$,
\[
 |\{ x \in X : g^k \cdot x = x \}| = f(\xi^k)
\]
where $g$ is a generator of $C$, and $\xi$ is a primitive $n^\thsup$ root of unity.
To this date, there many different instances of 
the cyclic sieving phenomenon has been proved,
see for example the survey by B.~Sagan~\cite{Sagan2011}.

Let $\SSYT(\lambda, n)$ denote the set of skew semi-standard Young tableaux
with shape $\lambda$ and maximal entry at most $n$.
Note that this set is empty unless $\length(\lambda)\leq n$.
Furthermore, let $n(\lambda)\coloneqq \sum_{j} (j-1) \lambda_j$.
The \defin{Schur polynomial} $\schurS_\lambda(\xvec)$
is defined as
\[
\schurS_\lambda(x_1,\dotsc,x_n) \coloneqq \sum_{T \in \SSYT(\lambda,n)} \xvec^T
\]
where $\xvec^T$ is the monomial $x_1^{w_1(T)} \dotsm x_n^{w_n(T)}$
and $w_i(T)$ is the number of entries in $T$ equal to $i$,
see \cite{StanleyEC2, Macdonald1995} for more background.
The symmetric group $\symS_n$ act on $\SSYT(\lambda,n)$
via type $A$ crystal reflection operators,
see M.~Kashiwaraand T.~Nakashima~\cite{KashiwaraNakashima1994} 
and M.~Shimozono's~\cite{Shimozono2005} survey for an excellent background.
The book \cite{BumpSchilling2015} provides a more thorough introduction.

Y.-T.~Oh and E.~Park~\cite{OhPark2019}, show that the triple
\begin{equation}\label{eq:OhPark}
\left(
 \SSYT(\lambda,n), \langle c_n \rangle, 
 q^{-n(\lambda)}\schurS_\lambda(1,q,q^2,\dotsc,q^{n-1})
 \right)
\end{equation}
exhibits the cyclic sieving phenomenon, 
where $c_n = \cryss_1 \cryss_2\dotsm \cryss_{n-1}$ is a product of crystal
reflection operators. Note that $c_n$ has order $n$.
In this note, we generalize \eqref{eq:OhPark} to skew shapes,
and refine the statement to smaller sets of tableaux, see \cref{thm:refinedCSP,cor:main}.
The CSP in \eqref{eq:OhPark} is in stark contrast with the conjectured CSP in
\cite{AlexanderssonAmini2018},
where we conjecture that there is a cyclic group $C$ of order $n$,
acting on $\SSYT(n\lambda/n\mu,m)$ and 
\[
\left(
 \SSYT(n\lambda/n\mu,m), C, \schurS_{n\lambda/n\mu}(1,q,q^2,\dotsc,q^{m-1})
 \right)
\]
is a CSP triple (with no restriction on $m$). 
Here, $n\lambda$ denotes element-wise multiplication $(n\lambda_1,\dotsc,n\lambda_\ell)$.

\subsection{Preliminaries}

Let $\Wcomp(m,n)$ denote the set of \defin{weak compositions},
which is the set of vectors $(\wc_1,\dotsc,\wc_n)$ in $\setN_{\geq 0}^n$
such that $\wc_1+\wc_2+\dotsb+\wc_n = m$.
Given a composition $\wc$, we let $|\wc|$ denote the sum of the entries.
We use the same notation for integer partitions,
and we let $|\lambda/\mu|$ denote the difference $|\lambda|-|\mu|$
whenever $\lambda/\mu$ is a skew shape.

Given a weak composition $\wc$, let $\SSYT(\lambda/\mu,\wc)$ denote 
the subset of $\SSYT(\lambda/\mu, n)$ where each tableau has 
exactly $\wc_i$ entries equal to $i$.
Note that this set is non-empty only if $|\lambda/\mu|=|\wc|$.
If $\mu = \emptyset$, we simply write $\SSYT(\lambda,\wc)$,
and $|\SSYT(\lambda,\wc)|$ is the \emph{Kostka coefficient} $K_{\lambda,\wc}$.
The \defin{skew Schur polynomial} $\schurS_{\lambda/\mu}(\xvec)$
is defined analogously to the Schur polynomials,
as a sum over skew SSYT:
\[
 \schurS_{\lambda/\mu}(x_1,\dotsc,x_n) \coloneqq \sum_{T \in \SSYT(\lambda/\mu,n)} \xvec^T
 = \sum_{\nu \vdash |\lambda/\mu| } K_{\lambda/\mu,\nu} \monom_{\nu}(x_1,\dotsc,x_n).
\]
For a weak composition $\wc = (\wc_1,\dotsc,\wc_n)$, let 
$\cyc_r(\wc)$ denote the \defin{cyclic shift}
$
 (\wc_{1+r},\wc_{2+r},\dotsc,\wc_{n+r})
$
where indices are computed modulo $n$. 
In particular, $\wc = \cyc_n(\wc)$.

\begin{lemma}
Let $\wc \in \Wcomp(m,n)$ and suppose that $\gcd(m,n)=1$. 
Then all cyclic shifts $\cyc_1(\wc)$, $\cyc_2(\wc)$, $\dotsc$, $\cyc_n(\wc)$
are different.
\end{lemma}
\begin{proof}
Suppose two cyclic shifts of $\wc$
are equal. Then 
$
\wc = (a_1,\dotsc,a_l,a_1,\dotsc,a_l,\dotsc, a_1,\dotsc,a_l)
$
where $kl=n$ and $k>1$.
Then $k(a_1+\dotsb+a_l)=m$ so it follows that $k|m$.
But then $k | \gcd(m,n)$, a contradiction.
\end{proof}

\begin{corollary}\label{cor:disjoint}
Let $\lambda/\mu$ be a skew shape and $\wc \in \Wcomp(|\lambda/\mu|,n)$
such that $\gcd(|\lambda/\mu|,n)=1$. 
Then the sets $\{ \SSYT(\lambda/\mu, \cyc_r(\wc)) \}_{r=1}^n$
are pairwise disjoint.
\end{corollary}

\begin{lemma}\label{lem:shiftedResidues}
Let $\wc \in \Wcomp(m,n)$ such that $\gcd(m,n)=1$.
Then $z_1,z_2,\dotsc,z_n$ defined via
 $z_r \coloneqq \sum_{j=1}^n j \wc_{j+r}$ for $r=1,2,\dotsc,n$
are all different mod $n$.
\end{lemma}
\begin{proof}
 It is easy to see that for all $r=1,2,\dotsc,n-1$, 
 we have the relation $z_{r+1} \equiv z_r + m \mod n$.
 Since $\gcd(m,n)=1$, the statement now follows from standard group theory.
\end{proof}

\begin{corollary}\label{cor:sum}
Let $\lambda/\mu$ be a skew shape and $\wc \in \Wcomp(|\lambda/\mu|,n)$
such that $\gcd(|\lambda/\mu|,n)=1$. 
Then 
\begin{align}\label{eq:blah}
 \sum_{r=1}^n \sum_{T \in \SSYT(\lambda/\mu,\cyc_r(\wc))} 
 q^{w_1(T)+w_2(T)+\dotsb + w_n(T)}  
\end{align}
is a multiple of $[n]_q = 1+q+q^2+\dotsb+ q^{n-1}$ $\mod q^n-1$.
\end{corollary}
\begin{proof}
By unraveling the definitions, \eqref{eq:blah} is equal to
 \[
 \sum_{r=1}^n  |\SSYT(\lambda/\mu,\cyc_r(\wc))| 
 q^{\wc_{1+r} + 2\wc_{2+r} + \dotsb + n \wc_{n+r}}
\]
where indices are taken modulo $n$.
since the cardinality of $\SSYT(\lambda/\mu,\cyc_r(\wc))$ does not depend on
the ordering of the entries in $\wc$, the expression \eqref{eq:blah} 
is equal to
\[
 |\SSYT(\lambda/\mu,\wc)| \sum_{r=1}^n q^{\wc_{1+r} + 2\wc_{2+r} + \dotsb + n \wc_{n+r}}.
\]
By \cref{lem:shiftedResidues}, the sum is now equal to $[n]_q$ mod $(q^n-1)$,
which proves the assertion.
\end{proof}

\subsection{Crystals and main proofs}

We have that $\symS_n$ act on $\SSYT(\lambda/\mu, n)$
via type $A$ \defin{crystal reflection operators}, 
see definition in M.~Shimozono's survey \cite{Shimozono2005}.
In particular, the simple reflection $\cryss_i \in \symS_n$
gives a bijection
\[
 \cryss_i : \SSYT(\lambda/\mu,\wc) \to \SSYT(\lambda/\mu, s_i \wc),
\]
where $s_i$ act on $\wc$ by interchanging entries at positions $i$ and $i+1$.
Let $c_n \coloneqq \cryss_1 \cryss_2\dotsm \cryss_{n-1} \in \symS_n$,
and let $\wc \in \Wcomp(|\lambda/\mu|,n)$.
One can then show that $c_n$ gives a bijection 
\begin{equation}\label{eq:crystal}
 c_n : \SSYT(\lambda/\mu,\wc) \to \SSYT(\lambda/\mu,\cyc_1(\wc))
\end{equation}
and $\langle c_n \rangle \subseteq \symS_n$ is a cyclic group of order $n$.
These properties of $c_n$ in \eqref{eq:crystal} are the only properties of crystal operators 
that we shall use in this note.
Note that the $c_n$-action is different from promotion obtained as a product of 
Bender--Knuth involutions, as promotion does not have order $n$ in general.

\begin{example}
Let $\lambda/\mu = 322/1$ and $\wc = (2,1,2,1)$.
Then 
\[
\ytableausetup{boxsize=1.0em}
 \left\{ 
 \ytableaushort{\none13,13,24}, 
 \ytableaushort{\none24,13,24}
 \right\}
 \quad \text{and}\quad
 \left\{ 
 \ytableaushort{\none14,13,24},
 \ytableaushort{\none13,12,24},
 \ytableaushort{\none23,13,24},
 \ytableaushort{\none24,13,34}
 \right\}
\]
are two orbits under $c_4$. 
Note that $\gcd(6,4)=2$ here, so \cref{cor:disjoint,lem:shiftedResidues,cor:sum}
do not apply. 
\end{example}

\begin{theorem}\label{thm:refinedCSP}
Let $\lambda/\mu$ be a skew shape with $m = |\lambda/\mu|$ 
and $\wc \in \Wcomp(|\lambda/\mu|,n)$ such that $\gcd(m,n)=1$.
Furthermore let $X$ be the (disjoint) union
\[
 X = \bigcup_{r=1}^n \SSYT(\lambda/\mu,\cyc_r(\wc)).
\]
Then the triple
\[
\left(
 X, \langle c_n \rangle, \sum_{T \in X} q^{w_1(T)+w_2(T)+\dotsb + w_n(T)}
\right)
\]
exhibit the cyclic sieving phenomenon.
\end{theorem}
\begin{proof}
By \cref{cor:disjoint}, it follows that 
every $c_n $-orbit in $X$ has size $n$.
By \cref{cor:sum}, we have that 
$f(q) \coloneqq \sum_{T \in X} q^{w_1(T)+w_2(T)+\dotsb + w_n(T)}$ is a multiple of $[n]_q$.
It follows that for all $k \in \setZ$,
\[
 |\{ T \in  X : (c_n)^k \cdot T = T \}| = f(\xi^k) 
\]
where $\xi$ is a $n^\thsup$ root of unity 
so the triple is indeed a CSP-triple.
\end{proof}

\begin{corollary}\label{cor:main}
Let $\gcd(|\lambda/\mu|,n)=1$. Then the triple
\begin{equation*}
 \left( \SSYT(\lambda/\mu,n), \langle c_n \rangle, 
 \schurS_{\lambda/\mu}(1,q,\dotsc,q^{n-1}) \right) 
\end{equation*}
exhibits the cyclic sieving phenomenon. 
Moreover, if $\mu=\emptyset$,
\begin{equation}\label{eq:ohPark2}
 \left( \SSYT(\lambda,n), \langle c_n \rangle, 
 q^{-n(\lambda)} \schurS_{\lambda}(1,q,\dotsc,q^{n-1}) \right) 
\end{equation}
is also a CSP-triple.
\end{corollary}
\begin{proof}
By summing over all possible compositions $\wc$ in  $\Wcomp(|\lambda/\mu|,n)$
in \cref{thm:refinedCSP}, we immediately have that
\[
 \left( \SSYT(\lambda/\mu,n), \langle c_n \rangle, 
 \schurS_{\lambda/\mu}(q,q^2,\dotsc,q^n) \right) 
\]
 is a CSP-triple and $\schurS_{\lambda/\mu}(q,q^2,\dotsc,q^n)$ evaluates to $0$
 for all $q=\xi,\xi^2,\dotsc,\xi^{n-1}$ whenever $\xi$ is a primitive $n^\thsup$
 root of unity. It is also straightforward to see that 
 $
 \schurS_{\lambda/\mu}(q,q^2,\dotsc,q^n) = 
 q^{|\lambda/\mu|} \schurS_{\lambda/\mu}(1,q,\dotsc,q^{n-1}),
 $
 so $\schurS_{\lambda/\mu}(1,q,\dotsc,q^{n-1})$ must evaluate to $0$
 for all $q=\xi,\xi^2,\dotsc,\xi^{n-1}$ as well.
 Furthermore, since entries in row $j$ in a non-skew semi-standard Young tableau 
 have value at least $j$, $q^{-n(\lambda)}\schurS_{\lambda}(1,q,\dotsc,q^{n-1})$ 
 is a polynomial in $q$. These observations prove the two statements.
 \end{proof}

We can actually say something much more general, 
where a cyclic sieving phenomenon can be constructed 
for any homogeneous, Schur-positive symmetric function.
\begin{corollary}
Let $f(\xvec) = \sum_{T \in X} \xvec^T$ be a homogeneous symmetric function
of degree $m$, such that the Schur expansion
$
 f(\xvec) = \sum_{\lambda \vdash m} d_{\lambda} \schurS_\lambda(\xvec)
$
is non-negative with $d_{\lambda} \in \setN$.
Furthermore, suppose $\gcd(m,n)=1$ and that there is a bijection $\psi$:
\[
\psi : X \to \bigcup_\lambda [d_\lambda]\times \SSYT(\lambda,n).
\]
Moreover, suppose that there is a cyclic group action $\langle c_n \rangle$ of order $n$ 
act on $X$ in such a way that it acts (via $\psi$) 
as the product of crystal reflections $c_n$ on $\SSYT(\lambda,n)$.
Then 
\[
 \left(X, \langle c_n \rangle, f(1,q,q^2,\dotsc,q^{n-1}) \right)
\]
is a CSP-triple, and every orbit under $c_n$ has size $n$.
\end{corollary}
Note that such a bijection $\psi$ and action $c_n$ can in principle be found for any 
Schur-positive symmetric function.
In particular, we can use the Littlewood--Richardson rule on 
$\schurS_{\lambda/\mu}(x_1,\dotsc,x_n)$, to express it as a non-negative 
sum of non-skew Schur polynomials and then apply previous corollary.

The previous corollary can then be applied to the following families of symmetric functions,
whose Schur expansion can be proved by an explicit bijection $\psi$,
and there is a type $A$ crystal structure on the underlying combinatorial objects:
 Modified Hall--Littlewood symmetric functions, \cite{KaliszewskiMorse2017},
 type $A$ and Stanley symmetric functions, \cite{MorseSchilling2015},
 type $C$ Stanley symmetric functions, \cite{HawkesParamonovSchilling2017},
 specialized  non-symmetric Macdonald polynomials, see \cite{AssafGonzalez2008},
 and (some) dual $k$-Schur functions, \cite{MorseSchilling2015}.

\bibliographystyle{alphaurl}
\bibliography{bibliography}

\end{document}